\documentclass[10pt]{article}
\usepackage{amsmath,amsfonts,amssymb,amsthm, bbm}
\usepackage[ansinew]{inputenc}
\usepackage{epsf,epsfig,psfrag}
\usepackage{color}
\usepackage[active]{srcltx}
\DeclareMathOperator{\dist}{dist}

\DeclareMathOperator{\diam}{diam}

\DeclareMathOperator{\round}{round}
\DeclareMathOperator{\id}{id}

\DeclareMathOperator{\co}{co}
\DeclareMathOperator{\Proj}{Proj}

\begin{document}

\newtheorem{oberklasse}{OberKlasse}
\newtheorem{lemma}[oberklasse]{Lemma}
\newtheorem{proposition}[oberklasse]{Proposition}
\newtheorem{theorem}[oberklasse]{Theorem}
\newtheorem{remark}[oberklasse]{Remark}
\newtheorem{corollary}[oberklasse]{Corollary}
\newtheorem{definition}[oberklasse]{Definition}
\newtheorem{example}[oberklasse]{Example}
\newtheorem{observation}[oberklasse]{Observation}
\newtheorem{assumption}{Assumption}
\newcommand{\clconvhull}{\ensuremath{\overline{\co}}}
\newcommand{\R}{\ensuremath{\mathbbm{R}}}
\newcommand{\N}{\ensuremath{\mathbbm{N}}}
\newcommand{\Z}{\ensuremath{\mathbbm{Z}}}
\newcommand{\Q}{\ensuremath{\mathbbm{Q}}}
\newcommand{\X}{\ensuremath{\mathcal{X}}}
\newcommand{\M}{\ensuremath{\mathcal{M}}}
\newcommand{\F}{\ensuremath{\mathcal{F}}}
\newcommand{\V}{\ensuremath{\mathcal{V}}}
\newcommand{\I}{\ensuremath{\mathcal{I}}}
\newcommand{\ES}{\ensuremath{\mathcal{S}}}
\newcommand{\ClSets}{\ensuremath{\mathcal{A}}}
\newcommand{\CpSets}{\ensuremath{\mathcal{C}}}
\newcommand{\CoCpSets}{\ensuremath{\mathcal{CC}}}
\newcommand{\powerset}{\ensuremath{\mathcal{P}}}
\newcommand{\reachable}{\ensuremath{\mathcal{R}}}
\newcommand{\fath}{\boldsymbol{h}}
\newcommand{\fatrho}{\boldsymbol{\rho}}
\newcommand{\fateps}{\boldsymbol{\epsilon}}

\newcommand{\tc}{\textcolor}
\newcommand{\mc}{\mathcal}

\renewcommand{\phi}{\ensuremath{\varphi}}
\renewcommand{\epsilon}{\ensuremath{\varepsilon}}

\title{Semi-implicit Euler schemes for ordinary differential inclusions}
\author{Janosch Rieger\\
Institut f\"ur Mathematik, Universit\"at Frankfurt\\
Postfach 111932, D-60054 Frankfurt a.M., Germany}
\date{\today }
\maketitle

\begin{abstract}
Two semi-implicit Euler schemes for differential inclusions are proposed and analyzed in depth.
An error analysis shows that both semi-implicit schemes inherit 
favorable stability properties from the differential inclusion.
Their performance is considerably better than that of the implicit Euler scheme, 
because instead of implicit inclusions only implicit equations have to be solved
for computing their images. 
In addition, they are more robust with respect to spatial discretization than
the implicit Euler scheme.
\end{abstract}

{\bf AMS(MOS) subject classifications.}  65L20, 34A60\\

\section{Introduction}

Consider the ordinary differential inclusion 
\begin{equation}
\label{ODI}
\dot x(t) \in F(t,x(t)),\ t\in[0,T], \quad x(0)=x_0\in\R^d.
\end{equation}
A solution of \eqref{ODI} is an absolutely continuous function $x:[0,T]\rightarrow\R^d$
which satisfies the inclusion almost everywhere. 
An approximation of the set $\ES([0,T],x_0)$ of all solutions of \eqref{ODI} is possible, but often 
too complex to be of practical value. For many problems, however, it suffices to compute the 
reachable sets
\begin{equation*}
\reachable(t,x_0) = \{x(t): x(\cdot)\in\ES([0,T],x_0)\}
\end{equation*}
for $t\in[0,T]$, which are the sets of all states that can be reached at time $t$ by an arbitrary 
solution starting from $x_0$ at time zero. Basic properties of the reachable set regarded as a set-valued 
mapping depending on initial value and end time are well-known (see \cite{Aubin:Cellina:84}). 

\medskip

One of the first attempts to approximate the set of solutions and the reachable set of \eqref{ODI} 
by an Euler-like scheme was presented in \cite{Dontchev:Farkhi:89}. 
A broad overview over classical literature with a focus
on Runge-Kutta schemes is presented in \cite{Lempio:Veliov:98}.
The numerical method proposed 
in \cite{Grammel:03} only uses extremal points of the right-hand side, so that a fully 
discretized scheme is obtained for right-hand sides with finitely many extremal points. 
A detailed analysis of spatial discretization effects is presented in \cite{Beyn:Rieger:07}.
Recently, error estimates for the set-valued Euler scheme were given for differential inclusions 
with state constraints (see \cite{Baier:Chahma:Lempio:07}) and non-convex differential
inclusions (see \cite{Sandberg:08}). First numerical methods for the approximation of 
solution sets of elliptic partial differential inclusions have been proposed in 
\cite{Beyn:Rieger:11:submitted} and \cite{Rieger:11}.

A set-valued implicit Euler scheme has been analyzed in \cite{Beyn:Rieger:10}. It has very good 
analytical properties, and it is based on an implicit function theorem that is given in 
\cite{Beyn:Rieger:10a}. If applied to stiff differential inclusions, it is considerably more efficient than
the explicit Euler scheme, because it senses the correct asymptotic behavior, while the reachable sets 
of the explicit Euler scheme do not only oscillate, but grow exponentially in diameter if the temporal
step-size is not small enough. The construction of the spatial discretization of the implicit Euler scheme, 
however, requires explicit knowledge of the one-sided 
Lipschitz constant and the modulus of continuity of the right-hand side and is very sensitive
to ill-estimated constants. This is the main motivation for the development of the semi-implicit
Euler schemes \eqref{Euler:1} and (\ref{direct:1}, \ref{direct:2}) analyzed in the present paper. 
In addition, their performance is considerably better than that of the implicit Euler scheme, 
because instead of implicit inclusions only implicit equations have to be solved for computing 
their images. 

Both semi-implicit schemes are no Runge-Kutta schemes in the sense of \cite{Lempio:Veliov:98}.

\section{Preliminaries}

Let the power set and the spaces of nonempty compact and nonempty convex and compact subsets of $\R^d$ be denoted by 
$\powerset(\R^d)$, $\CpSets(\R^d)$ and $\CoCpSets(\R^d)$.
The Euclidean norm is denoted by $|\cdot|$, while
$\|A\|:=\sup_{a\in A}|a|$ denotes the maximal norm of the elements of a set $A\in\CpSets(\R^d)$.
For $A,B\in\CpSets(\R^d)$, the one-sided and the symmetric Hausdorff distance are given by
\begin{align*}
\dist(A,B) &:= \sup_{a\in A}\inf_{b\in B}|a-b|\\
\dist_H(A,B) &:= \max\{\dist(A,B),\dist(B,A)\}.
\end{align*}
For $A\in\CpSets(\R^d)$ and $r>0$ denote $B_r(A):=\{x\in\R^d: \dist(x,A)\le r\}$. If $A=\{x\}$ is
a point, $B_r(x)$ is simply the closed ball with radius $r$ centered at $x$.

For a set-valued mapping $F:\R^d\rightarrow\CpSets(\R^d)$ and $A\subset\R^d$ define $F(A):=\cup_{a\in A}F(a)$.
Thus the composition $F\circ G$ of two set-valued mappings $F$ and $G$ is given by
$(F\circ G)(x):=\cup_{y\in G(x)}F(y)$, and $F^k(x):=(F\circ\ldots\circ F)(x)$ is defined by induction.
The projection $\Proj:\R^d\times\CpSets(\R^d)\rightarrow\CpSets(\R^d)$ is defined by
$\Proj(x,A):=\{a\in A: |x-a|\le|x-a'|\ \text{for all}\ a'\in A\}$. If $A\in\CoCpSets(\R^d)$,
then $\Proj(x,A)$ is a singleton.

A set-valued mapping $F:\R^d\rightarrow\CpSets(\R^d)$ is called continuous if it is continuous w.r.t.\ 
the Euclidean metric and the Hausdorff distance and $L$-Lipschitz continuous if
\[\dist_H(F(x),F(x')) \le L|x-x'|\ \text{for all}\ x,x'\in\R^d.\]

The notion of relaxed one-sided Lipschitz set-valued mappings generalizes the concepts of Lipschitz 
continuity and the (strong) one-sided Lipschitz property. A detailed analysis of this property can be found 
in \cite{Donchev:02} and several other works of this author.
\begin{definition}
\label{ROSL}
A mapping $F:\R^m\rightarrow\CoCpSets(\R^d)$ is called relaxed one-sided Lipschitz with constant
$l\in\R$ (or $l$-ROSL) if for every $x,x'\in\R^d$ and $y\in F(x)$ there exists some $y'\in F(x')$ such that
\begin{equation}
\langle y-y',x-x' \rangle \leq l|x-x'|^2.
\end{equation}
\end{definition}
A single-valued function $f:\R^d\rightarrow\R^d$ is $l$-ROSL if and only if it is one-sided Lipschitz with constant 
$l$ (or $l$-OSL) in the classical sense, i.e.\ if it satisfies
\[\langle f(x)-f(x'),x-x'\rangle \le l|x-x'|^2\]
for all $x,x'\in\R^d$.

\begin{lemma}
If $(\Omega,\mathcal{A},\mu)$ is a measurable space with $\mu(\Omega)=1$ and $g:\Omega\rightarrow\R^d$
is a $\mu$-integrable function, then for every $A\in\CoCpSets(\R^d)$, the estimate
\[\dist(\int_\Omega g(\omega)d\mu(\omega),A) \le \int_\Omega \dist(g(\omega),A)d\mu(\omega)\]
holds.
\end{lemma}

\begin{proof}
Apply Jensen's inequality to the convex and Lipschitz continuous function $\dist(\cdot,A):\R^d\rightarrow\R$.
\end{proof}

Throughout this paper, it will be assumed that there exists a splitting 
\begin{equation}
\label{splitting}
F(t,x) = f(t,x) + M(t,x),
\end{equation}
where $f:[0,T]\times\R^d\rightarrow\R^d$ and $M:[0,T]\times\R^d\rightarrow\CoCpSets(\R^d)$
satisfy the following properties.
\begin{itemize}
\item [A1)] The function $(t,x) \mapsto f(t,x)$ is continuous and there exists a continuous 
integrable function $l_f:[0,T]\rightarrow\R$ such that the mapping
$x \mapsto f(t,x)$ is OSL with constant $l_f(t)$ for any $t$.
\item [A2)] The mapping $(t,x) \mapsto M(t,x)$ is continuous w.r.t.\ the Hausdorff
metric and there exists a continuous integrable function $L_M:[0,T]\rightarrow\R_+$ such that the 
multimap $x \mapsto M(t,x)$ is Lipschitz with constant $L_M(t)$ for every $t$. 
\end{itemize}
As a consequence, the set-valued mapping $F$ is jointly continuous, and $x\mapsto F(t,x)$ is
$l(t)$-ROSL with $l(t)=l_f(t)+L_M(t)$.

\medskip

In view of the following remark, the above assumptions are not unreasonable.
\begin{remark}
A control system with affine linear controls has the form
\[\dot x(t) = f(t,x(t)) + A(t,x(t))u(t), \quad u(t)\in U\]
with a function $f:[0,T]\times\R^d\rightarrow\R^d$, a matrix-valued mapping 
$A:[0,T]\times\R^d\rightarrow\R^{d \times m}$ and 
a convex and compact control set $U\subset\R^m$.
If $f$ satisfies A1) and $A$ is continuous and $L_A$-Lipschitz in the
second argument, then the right-hand side of the corresponding differential inclusion
\[\dot x(t) \in f(t,x(t)) + A(t,x(t))U\]
exhibits a splitting of type \eqref{splitting}, where $M(t,x) := A(t,x)U$ is jointly 
continuous and $L_A\|U\|$-Lipschitz (and hence $L_A\|U\|$-ROSL) in the second argument.
In \cite[Theorem 4.1]{Lempio:Veliov:98}, Runge-Kutta schemes are analyzed in this setting.
The results there focus on the order of convergence of these schemes and not on
stiffness and asymptotic behavior.

Differential equations with uncertainty are a special case of the above control system
with $m=d$, $A(t,x)=r(t,x)\cdot\id$, where $r:[0,T]\times\R^d\rightarrow\R_+$ is a non-negative
real-valued function, and $U=B_1(0)$.

\medskip

A set-valued semi-implicit Euler scheme has been considered, but not thoroughly analyzed
in \cite{Kloeden:Valero:09}, where attractors of the partial differential inclusion
\[u_t \in \Delta u + F(u)\]
with Dirichlet boundary condition and multivalued nonlinearity $F$ and its Galerkin approximations
\[\frac{d}{dt}u^{(N)} \in \Delta^{(N)} u^{(N)} + F^{(N)}(u^{(N)})\]
are investigated. As all eigenvalues of the discretized Laplace operator $\Delta^{(N)}$ are
negative, the Galerkin differential inclusion exhibits a splitting of the right-hand side
into a OSL component $\Delta^{(N)} u^{(N)}$ with negative Lipschitz constant and a 
Lipschitz continuous nonlinearity $F^{(N)}(u^{(N)})$. Hence
the semi-implicit Euler scheme
\[u_{n+1}^{(N)} \in u_{n}^{(N)} + h\Delta^{(N)}u_{n+1}^{(N)} + hF^{(N)}(u_n^{(N)})\]
is solvable for any step-size $h>0$ with solution
\[u_{n+1}^{(N)} \in (I-h\Delta^{(N)})^{-1}[u_{n}^{(N)} + h F^{(N)}(u_n^{(N)})].\]
\end{remark}

\section{The parameterized semi-implicit Euler scheme}
\label{parametrized:sec}

In this section, the parameterized version of the semi-implicit Euler scheme will be analyzed.
One step of this method is a multivalued mapping $\Phi:D\rightarrow\powerset(\R^d)$ given by
\begin{equation}
\label{Euler:1}
\Phi(t,x;h) = \{z\in\R^d: z \in x + hf(t+h,z) + hM(t,x)\},
\end{equation}
where $D\subset[0,T]\times\R^d\times\R_+$ is a suitable domain of definition (see Section \ref{single:step:properties}).
The scheme $\Phi$ is implicit in the single-valued part $f$ of $F$ and explicit in the set-valued component $M$.
It is called the parameterized semi-implicit Euler scheme, because its solution set can be parameterized over the set $M(t,x)$
as shown in Lemma \ref{representation:lemma}.

Approximations of the reachable sets $\reachable(t,x_0)$, $t\in[0,T]$, of \eqref{ODI} are obtained by iterating the scheme $\Phi$ 
on a temporal grid. Define $\Delta_{\fath}=\{t_0,t_1,\ldots,t_N\}$, where $N\in\N$, $0=t_0<t_1<\ldots<t_N=T$,
$h_n:=t_{n+1}-t_n$ for $n=0,\ldots,N-1$ and $\fath:=(h_0,\ldots,h_{N-1})$. 
Throughout this paper, the following condition will be imposed on $\Delta_{\fath}$.
\begin{itemize}
\item [A3)] For any $n\in\{0,\ldots,N-1\}$, the step-size $h_n$ satisfies $l_f(t_{n+1})h_n<1$.
\end{itemize}
In view of Theorem \ref{properties:Phi}, this assumption guarantees that the iterates of the 
numerical scheme are well-defined.

\begin{definition}
A sequence $\{y_n\}_{n=0}^{N}\subset\R^d$ is called a trajectory of the parameterized 
semi-implicit Euler scheme $\Phi$ associated with \eqref{ODI} if 
\begin{equation}
\label{def:trajectory:1}
y_{n+1} \in \Phi(t_n,y_n;h_n)\ \text{for}\ n=0,\ldots,N-1,\quad y_0=x_0.
\end{equation}
The set of all such trajectories is denoted by $\ES_\Phi(\Delta_{\fath},x_0)$.
\end{definition}

As in continuous time, the reachable set $\reachable_\Phi(t_n,x_0)$
is the set of all states that can be reached at time $t_n$ by a discrete trajectory starting from $x_0$ at time zero.
\begin{definition}
For any $t_n\in\Delta_{\fath}$, the reachable set $\reachable_\Phi$ of the parameterized semi-implicit Euler scheme $\Phi$ is given by
\begin{equation}
\label{def:reachable:1}
\reachable_\Phi(t_n,x_0)=\{y_n: \{y_k\}_{k=0}^N\in\ES_\Phi(\Delta_{\fath},x_0)\}.
\end{equation}
\end{definition}
According to Theorem \ref{properties:Phi}, the reachable set $\reachable_\Phi(t_n,x_0)$
depends continuously on $(\fath,x_0)$ and is Lipschitz continuous in the initial value,
because it is a composition of multimaps with this property.

\medskip

The properties of $\Phi$ regarded 
as a set-valued mapping depending on time, space, and step-size are investigated in 
Section \ref{single:step:properties}, while Section \ref{single:step:dynamics} deals
with the dynamics and convergence of the solution sets. In Section \ref{single:step:discretization},
the spatial discretization of the parameterized semi-implicit Euler scheme is analyzed, which is much easier 
and more robust than that of the fully implicit Euler scheme presented in \cite{Beyn:Rieger:10}.
All estimates in Sections \ref{single:step:dynamics} and \ref{single:step:discretization}
are given in terms of the solution sets, because these imply the same estimates for the reachable
sets.

\subsection{Properties of the scheme as a multivalued mapping}
\label{single:step:properties}

Define $D:=\{(t,x,h)\in[0,T]\times\R^d\times\R_+: l_f(t+h)h<1\}$.
\begin{lemma}
\label{D:open}
The set $D$ is relatively open in $[0,T]\times\R^d\times\R_+$.
\end{lemma}
\begin{proof}
Fix $(t,x,h)\in D$. Since $l_f(t+h)h<1$ and $(t',x',h')\mapsto l_f(t'+h')h'$
is continuous, there exists some $R>0$ such that $l_f(t'+h')h'<1$ for all 
$(t',x',h')\in B_R(t,x,h)\cap([0,T]\times\R^d\times\R_+)$, which
is a relatively open neighborhood of $(t,x,h)$.
\end{proof}

The set-valued mapping $G:\R^d\rightarrow\CoCpSets(\R^d)$ defined by
\[G_{t,x,h}(z) := x+hf(t+h,z)-z+hM(t,x)\]
is continuous and ROSL with constant $l_{G_{t,x,h}}:=-(1-l_f(t+h)h)$. 
The parameterized semi-implicit Euler scheme can be represented as 
\[\Phi(t,x;h) = \{z\in\R^d: z\in x+hf(t+h,z)+hM(t,x)\} = S_{G_{t,x,h}}(0),\]
where $S_{G_{t,x,h}}(0):=\{z\in\R^d:0\in G_{t,x,h}(z)\}$.

\begin{theorem}
\label{properties:Phi}
For any $(t,x,h)\in D$, the set $\Phi(t,x;h)$ is nonempty and compact for each $x\in\R^d$
and satisfies
\begin{equation}
\label{diameter:Phi}
\diam \Phi(t,x;h) \le \frac{h}{1-l_f(t+h)h}\sup_{\xi\in\Phi(t,x;h)}\diam M(t,\xi).
\end{equation}
For any $y\in\R^d$, the scheme satisfies the estimate
\begin{align}
\dist(y,\Phi(t,x;h)) \le \frac{1}{1-l_f(t+h)h} \dist(y,x+hf(t+h,y)+hM(t,x)), \label{defect:estimate:Phi}\\
\dist(\Phi(t,x;h),y) \le \frac{1}{1-l_f(t+h)h} \dist(x+hf(t+h,y)+hM(t,x),y). \label{distance:estimate:Phi}
\end{align}
\end{theorem}

\begin{proof}
By Theorem \ref{solvability:theorem}, $\Phi(t,x;h) = S_{G_{t,x,h}}(0)$ is nonempty and compact and
satisfies \eqref{diameter:Phi}. Moreover, \eqref{defect:estimate:Phi} and \eqref{distance:estimate:Phi}
are implied by Theorem \ref{solvability:theorem}, because for any $y\in\R^d$,
\begin{align*}
\dist(y,\Phi(t,x;h)) &= \dist(y,S_{G_{t,x,h}}(0)) \le \frac{1}{1-l_f(t+h)h}\dist(0,G_{t,x,h}(y))\\
&\le \frac{1}{1-l_f(t+h)h} \dist(y,x+hf(t+h,y)+hM(t,x)),\\
\dist(\Phi(t,x;h),y) &= \dist(S_{G_{t,x,h}}(0),y) \le \frac{1}{1-l_f(t+h)h}\dist(G_{t,x,h}(y),0)\\
&\le \frac{1}{1-l_f(t+h)h} \dist(x+hf(t+h,y)+hM(t,x),y).
\end{align*}
\end{proof}

\begin{corollary}
The iterates of the parameterized semi-implicit Euler scheme given by \eqref{def:trajectory:1} and \eqref{def:reachable:1}
are well-defined.
\end{corollary}
\begin{proof}
Assumption A3) implies that $(t_n,x,h_n)\in D$ for any $x\in\R^d$, and hence the defining implicit 
inclusion \eqref{Euler:1} is solvable in every step according to Theorem \ref{properties:Phi}.
\end{proof}

The following lemma will be useful for the discussion of connectedness and convexity of the images of $\Phi$.
\begin{lemma}
\label{representation:lemma}
For fixed $(t,x,h)\in D$, the image of the parameterized semi-implicit Euler scheme can be represented as 
\begin{equation}
\label{implicit:union}
\Phi(t,x;h) = \cup_{m\in M(t,x)}z(m),
\end{equation}
where $z(m)$ is the unique solution $z$ of the implicit equation 
\begin{equation}
\label{implicit:equation}
z = x+hf(t+h,z)+hm.
\end{equation}
Moreover, the mapping $z:M(t,x)\rightarrow\R^d$ is $\frac{h}{1-l_f(t+h)h}$-Lipschitz.
\end{lemma}

\begin{proof}
Theorem \ref{solvability:theorem} applied to the equation 
\begin{equation}
\label{local:6}
0=x+hf(t+h,z)-z+hm
\end{equation}
guarantees the
existence of a solution of \eqref{implicit:equation}. Assume that \eqref{implicit:equation} has two
solutions $z$ and $z'$. Then $z-z'=h(f(t+h,z)-f(t+h,z'))$, and by the OSL property,
\[|z-z'|^2 = h\langle f(t+h,z)-f(t+h,z'),z-z'\rangle \le l_f(t+h)h|z-z'|^2,\]
which forces $|z-z'|^2=0$, because $l_f(t+h)h<1$.

If $z,z'\in\Phi(t,x;h)$, then there exist $m,m'\in M(t,x)$ such that $z$ and $z'$ satisfy
\eqref{implicit:equation} with $m$ and $m'$. Theorem \ref{solvability:theorem} applied to \eqref{local:6} 
ensures that
\begin{align*}
&|z(m)-z(m')|\\ 
&\le -l_{G_{t,x,h}}^{-1}|x+hf(t+h,z(m'))-z(m')+hm| \\
&\le -l_{G_{t,x,h}}^{-1}|x+hf(t+h,z(m'))-z(m')+hm'| + -l_{G_{t,x,h}}^{-1}h|m-m'|\\
&= -l_{G_{t,x,h}}^{-1}h|m-m'|.
\end{align*}
\end{proof}

\begin{corollary}
For fixed $(t,x,h)\in D$, the set $\Phi(t,x;h)$ is path-connected.
\end{corollary}

\begin{proof}
The set $\Phi(t,x;h)=z(M(t,x))$ is path-connected, because $M(t,x)$ is path-connected 
and $z:M(t,x)\rightarrow\R^d$ is continuous.
\end{proof}

\begin{corollary}
If $d=1$, then the set $\Phi(t,x;h)$ is a convex interval.
\end{corollary}

\begin{proof}
Every compact connected subset of $\R^1$ is a convex interval.
\end{proof}

The images of the parameterized semi-implicit Euler scheme are not necessarily convex for $d>1$.
\begin{example}
\label{not:convex}
Consider the mapping $F:\R^2\rightarrow\R^2$ given by
\[F(x_1,x_2) = f(x_1,x_2) + M(x_1,x_2) = \frac12\begin{pmatrix}-x_1-x_2\\ x_1-x_2-x_2^3\end{pmatrix} 
+ \begin{pmatrix} [0,\frac{13}{2}]\\0 \end{pmatrix}.\]
Because of
\begin{align*}
\langle \begin{pmatrix} x_1\\x_2 \end{pmatrix} ,f(x_1,x_2)\rangle
&= \frac12\langle\begin{pmatrix}x_1\\x_2\end{pmatrix},\begin{pmatrix}-x_1-x_2\\ x_1-x_2-x_2^3\end{pmatrix}\rangle\\
&=-\frac12(x_1^2+x_2^2+x_2^4) \le -\frac12(x_1^2+x_2^2),
\end{align*}
the mappings $f$ and $F$ are $-\frac12$-ROSL.
According to Lemma \ref{representation:lemma}, the image $\Phi(0,0,1)$ is a Lipschitz continuous curve 
parameterized over $M(0,0)=([0,\frac{13}{2}],0)^T$ and hence over
the interval $[0,\frac{13}{2}]$. Solving \eqref{implicit:equation} for $m_0=0$, $m_1=\frac{43}{16}$
and $m_2=\frac{13}{2}$ yields $z(m_0)=(0,0)^T$, $z(m_1)=(\frac{13}{8},\frac12)^T$ and $z(m_2)=(4,1)^T$.
As these three points are not on a straight line (see Figure \ref{eulerbild}), 
the image $\Phi(0,0,1)$ is not convex.
\end{example}

\begin{figure}
\begin{center}
\includegraphics[scale=0.5]{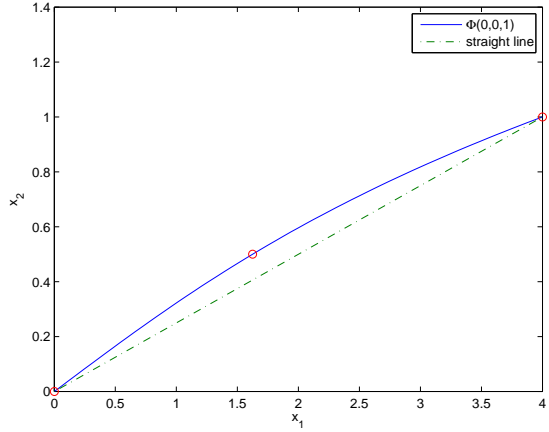}
\end{center}
\caption{The image $\Phi(0,0,1)$ is not convex for the right-hand side $F$ from Example \ref{not:convex}. 
The highlighted points correspond to the three exact solutions given there. \label{eulerbild}}
\end{figure}

\begin{theorem}
\label{Euler:1:continuous}
The set-valued mapping $(t,x,h)\mapsto\Phi(t,x;h)$ is jointly continuous on $D$,
and $x\mapsto\Phi(t,x;h)$ is $\frac{1+hL_M(t)}{1-l_f(t+h)h}$-Lipschitz. Moreover,
$\Phi(t,x;0)=x$ for any $t\in\R$.
\end{theorem}

\begin{proof}
Fix any $(t,x,h)\in D$. By Lemma \ref{D:open}, there exists some 
$R>0$ such that $D_R:=B_R(t,x,h)\cap([0,T]\times\R^d\times\R_+)\subset D$ and hence $D_R=B_R(t,x,h)\cap D$.
By Theorem \ref{properties:Phi},
\begin{equation}
\label{local:1}
\|\Phi(t',x';h')\| \le \frac{1}{1-l_f(t'+h')h'}\|x'+h'f(t'+h',0)+h'M(t',x')\|
\end{equation}
for any $(t',x',h')\in D_R$.
Since $D_R$ is compact and the right-hand side of \eqref{local:1} is continuous in $(t',x',h')$,
there exists some $C>0$ such that 
\[\|\Phi(t',x';h')\| \le C\ \text{for all}\ (t',x',h')\in D_R,\]
so that $\Phi(D_R)\subset B_C(0)$. For any $z'\in\Phi(t',x';h')$ with $(t',x',h')\in D_R$, 
\begin{align}
\label{local:2}
&\dist(z',\Phi(t,x;h)) \nonumber\\
&\le -l_{G_{t,x,h}}^{-1} \dist(0,G_{t,x,h}(z')) \nonumber \\
&\le -l_{G_{t,x,h}}^{-1} \big( \dist(0,G_{t',x',h'}(z')) + \dist(G_{t',x',h'}(z'),G_{t,x,h}(z')) \big)\\
&\le -l_{G_{t,x,h}}^{-1} \big(|x'-x| + \dist(h'M(t',x'),hM(t,x))\big) \nonumber\\
&\quad -l_{G_{t,x,h}}^{-1} |h'f(t'+h',z')-hf(t+h,z')| =: \phi(t',x',h',z') \nonumber
\end{align}
according to Theorem \ref{solvability:theorem}.
The function $\phi:D_R\times B_C(0)\rightarrow\R_+$ is continuous, and for any $z'\in B_C(0)$,
$\phi(t',x',h',z')\rightarrow 0$ as $(t',x',h')\rightarrow (t,x,h)$.
Assume that 
\[\sup_{z'\in B_C(0)} \phi(t',x',h',z') \nrightarrow 0\ \text{as}\ (t',x',h')\rightarrow (t,x,h).\]
Then there exists some $\epsilon>0$ and a sequence $(t_n',x_n',h_n',z_n')_{n\in\N}\subset D_R\times B_C(0)$,
such that $\phi(t_n',x_n',h_n',z_n')>\epsilon$ and $(t_n',x_n',h_n')\in B_{1/n}(t,x,h)$. As $B_C(0)$
is compact, there exists a subsequence $(t_n',x_n',h_n',z_n')_{n\in\N'}$, $\N'\subset\N$, and some $z_0\in B_C(0)$ 
such that $z_n'\rightarrow z_0$ as $n'\rightarrow\infty$. But then
\[\phi(t_n',x_n',h_n',z_n') \rightarrow \phi(t,x,h,z_0) = 0\]
is a contradiction, so that
\[\sup_{z'\in B_C(0)} \phi(t',x',h',z') \rightarrow 0\ \text{as}\ (t',x',h')\rightarrow (t,x,h).\]
Hence
\[\dist(\Phi(t',x';h'),\Phi(t,x;h)) = \sup_{z'\in\Phi(t',x';h')}\dist(z',\Phi(t,x;h))\rightarrow 0.\]
The proof of the statement $\dist(\Phi(t,x;h),\Phi(t',x';h')) \rightarrow 0$ is almost identical, 
because the set $\Phi(t,x;h)$ is compact according to Theorem \ref{properties:Phi}.
Lipschitz continuity of $\Phi$ in $x$ follows from \eqref{local:2} with $t=t'$ and $h=h'$.
\end{proof}

\subsection{Dynamic properties and convergence analysis}
\label{single:step:dynamics}

The solutions of the differential inclusion and the parameterized semi-implicit Euler scheme are uniformly bounded.
\begin{lemma}
\label{boundedness:lemma}
There exists some constant $C>0$ such that every solution $x(\cdot)$ of the differential inclusion
\eqref{ODI} and every trajectory $\{y_n\}_n$ of the parameterized semi-implicit Euler scheme with $y_0=x_0$ 
are contained in $B_C(0)\subset\R^d$.
\end{lemma}
\begin{proof}
Consider an arbitrary solution $x(\cdot)$ of \eqref{ODI}.
By the ROSL property, for every $t\in[0,T]$ there exists an element $v(t)\in F(t,0)$ such that
\[\langle\dot x(t)-v(t),x(t)\rangle \le l(t)|x(t)|^2,\]
and hence
\[ |x(t)|\frac{d}{dt}|x(t)| = \langle\dot x(t)-v(t),x(t)\rangle + \langle v(t),x(t)\rangle
\le l(t)|x(t)|^2 + |v(t)||x(t)|,\]
so that 
\[\frac{d}{dt}|x(t)| \le l(t)|x(t)| + |v(t)| \le l(t)|x(t)| + \|F(t,0)\|\]
whenever $x(t)\neq 0$. Applying the Gronwall Lemma on every subinterval $[t_b,t_e]\subset[0,T]$
with $|x(t)|\ge 1$ for any $t\in[t_b,t_e]$ yields the desired bound.

Boundedness of the discrete trajectories is shown by induction.  By assumption, $y_0=x_0\in B_{|x_0|}(0)$.
Assume that there exists some $R_k>0$ such that $y_k\in B_{R_k}(0)$ for all discrete trajectories $\{y_n\}_n$.
Then estimate \eqref{distance:estimate:Phi} from Theorem \ref{properties:Phi} yields that
\begin{align*}
|y_{k+1}| &\le \|\Phi(t_k,y_k;h_{k})\| \le \frac{1}{1-l_f(t_{k+1})h_k} \|y_k+h_kf(t_{k+1},0)+h_kM(t_k,y_k)\| \\
&\le \frac{1}{1-l_f(t_{k+1})h_k} \sup_{y\in B_{R_k}(0)}\|y+h_kf(t_{k+1},0)+h_kM(t_k,y)\|=: R_{k+1} < \infty,
\end{align*}
because $B_{R_k}(0)$ is compact and the continuous function $y\mapsto y+h_kM(t_k,y)$ is bounded on $B_{R_k}(0)$. 
An iteration of this argument shows that any trajectory $\{y_n\}_n$ is contained in
$B_R(0)$ with $R:=\max_{k=0,\ldots,N}R_k$.
\end{proof}

As a consequence, the ODI \eqref{ODI} and the parameterized semi-implicit Euler scheme \eqref{Euler:1} can be considered on 
the compact set $[0,T]\times B_C(0)$, where $f$ and $M$ are uniformly continuous so that the moduli 
of continuity $\tau_f,\chi_f,\tau_M:\R_+\rightarrow\R_+$ given by
\begin{align*}
\tau_f(\delta)&:=\sup\{|f(t,x)-f(t',x)|: t,t'\in[0,T], |t-t'|\le\delta, x\in B_C(0)\},\\
\chi_f(\delta)&:=\sup\{|f(t,x)-f(t,x')|: t\in[0,T], x,x'\in B_C(0), |x-x'|\le\delta\},\\
\tau_M(\delta)&:=\sup\{\dist(M(t,x),M(t',x)): t,t'\in[0,T], |t-t'|\le\delta, x\in B_C(0)\}
\end{align*}
are well-defined. The error term
\[\Gamma(h,t) := \tau_f(h) + \chi_f(Ph) + \tau_M(h) + L_M(t)Ph\]
with $P:=\max_{x\in B_C(0)}|f(x)|+\max_{x\in B_C(0)}\|M(x)\|$ will appear 
frequently in the following discussion. Note that
\begin{equation}
\label{gamma:to:zero}
\|\Gamma(h,\cdot)\|_\infty \rightarrow 0\ \text{as}\ h\rightarrow 0,
\end{equation}
because $L_M$ is bounded.

\medskip

Proposition \ref{discrete:Filippow:1} is an existence and stability result for solutions 
of the semi-implicit Euler scheme \eqref{Euler:1} and therefore a discrete counterpart
of Theorem \ref{continuous:filippov}.

\begin{proposition}
\label{discrete:Filippow:1}
For any sequence $\{x_n\}_n\subset\R^d$, there exists a solution $\{y_n\}_n$ of the parameterized
semi-implicit Euler scheme \eqref{Euler:1} satisfying
\begin{align}
\label{local:4}
|x_n-y_n| &\le e^{\sum_{k=0}^{n-1}(\frac{l_f(t_{k+1})h_k}{1-l_f(t_{k+1})h_k}+L_M(t_k)h_k)}|x_0-y_0| \nonumber\\
&\quad + \sum_{k=0}^{n-1}e^{\sum_{j=k}^{n-1}\frac{l_f(t_{j+1})h_j}{1-l_f(t_{j+1})h_j}+\sum_{j=k+1}^{n-1}L_M(t_j)h_j}h_kg_k,
\end{align}
where $g_n:=\dist(\frac{1}{h_n}(x_{n+1}-x_n), f(t_{n+1},x_{n+1})+M(t_n,x_n))$.
\end{proposition}
\begin{proof}
Let the sequence $\{x_n\}_{n=0}^N$ be given, and assume that a trajectory $\{y_n\}_{n=0}^{\bar n}$ of 
\eqref{Euler:1} has already been constructed 
for some $\bar n<N$. By \eqref{defect:estimate:Phi}, there exists some
$y_{\bar n+1}\in\Phi(t_{\bar n},y_{\bar n};h_{\bar n})$ such that 
\begin{align*}
&|x_{{\bar n}+1}-y_{{\bar n}+1}|\\ 
&\le \frac{1}{1-l_f(t_{{\bar n}+1})h_{\bar n}}
\dist(x_{{\bar n}+1},y_{\bar n}+h_{\bar n}f(t_{{\bar n}+1},x_{{\bar n}+1})+h_{\bar n}M(t_{\bar n},y_{\bar n}))\\
&= \frac{1}{1-l_f(t_{{\bar n}+1})h_{\bar n}}
\dist(x_{{\bar n}+1}-x_{\bar n}+x_{\bar n}-y_{\bar n},h_{\bar n}f(t_{{\bar n}+1},x_{{\bar n}+1})+h_{\bar n}M(t_{\bar n},y_{\bar n}))\\
&\le \frac{1}{1-l_f(t_{{\bar n}+1})h_{\bar n}}|x_{\bar n}-y_{\bar n}| \\
&\quad + \frac{h_{\bar n}}{1-l_f(t_{{\bar n}+1})h_{\bar n}}
\dist(\frac{1}{h_{\bar n}}(x_{{\bar n}+1}-x_{\bar n}), f(t_{{\bar n}+1},x_{{\bar n}+1})+M(t_{\bar n},x_{\bar n}))\\
&\quad + \frac{h_{\bar n}}{1-l_f(t_{{\bar n}+1})h_{\bar n}}\dist(M(t_{\bar n},x_{\bar n}),M(t_{\bar n},y_{\bar n}))\\
&\le \frac{1+L_M(t_{\bar n})h_{\bar n}}{1-l_f(t_{{\bar n}+1})h_{\bar n}}|x_{\bar n}-y_{\bar n}| 
+ \frac{h_{\bar n}}{1-l_f(t_{{\bar n}+1})h_{\bar n}} g_{\bar n},
\end{align*}
so that by induction
\begin{align*}
&|x_n-y_n| \\
&\le \prod_{k=0}^{n-1}\frac{1+L_M(t_n)h_n}{1-l_f(t_{n+1})h_n}|x_0-y_0| 
+ \sum_{k=0}^{n-1}\big(\prod_{j=k+1}^{n-1} \frac{1+L_M(t_j)h_j}{1-l_f(t_{j+1})h_j}\big)\frac{h_k}{1-l_f(t_{k+1})h_k}g_k
\end{align*}
for all $n\in\{0,\ldots,N\}$.
The well-known formulas $\frac{1}{1-lh}=1+\frac{lh}{1-lh}$ for $lh<1$ 
and $1+t\le e^t$ for $t\in\R$ yield the desired statement.
\end{proof}

One-sided error estimates are a simple consequence of Proposition \ref{discrete:Filippow:1}.
\begin{corollary}
\label{convergence:Euler:1:1}
For every solution $x(\cdot)$ of the differential inclusion \eqref{ODI}, there exists a trajectory
$\{y_n\}_n$ of the semi-implicit Euler scheme \eqref{Euler:1} such that
\begin{equation}
\label{local:5}
|x(t_n)-y_n| \le \sum_{k=0}^{n-1}
e^{\sum_{j=k}^{n-1}\frac{l_f(t_{j+1})h_j}{1-l_f(t_{j+1})h_j}+\sum_{j=k+1}^{n-1}L_M(t_j)h_j}h_k\Gamma(h_k,t_k)
\end{equation}
for $n=0,\ldots,N$.
\end{corollary}
\begin{proof}
The statement follows from Proposition \ref{discrete:Filippow:1} with $x(t_0)=x_0=y_0$ and 
\begin{align*}
g_n &= \dist(\frac{1}{h_n}(x(t_{n+1})-x(t_n)),f(t_{n+1},x(t_{n+1}))+M(t_n,x(t_n)))\\
&= \dist(\frac{1}{h_n}\int_{t_n}^{t_{n+1}}\dot x(t)dt,f(t_{n+1},x(t_{n+1}))+M(t_n,x(t_n)))\\
&\le \frac{1}{h_n}\int_{t_n}^{t_{n+1}}\dist(\dot x(t)dt,f(t_{n+1},x(t_{n+1}))+M(t_n,x(t_n)))dt\\
&\le \frac{1}{h_n}\int_{t_n}^{t_{n+1}}\dist(f(t,x(t))+M(t,x(t)),f(t_{n+1},x(t_{n+1}))+M(t_n,x(t_n)))dt\\
&\le \Gamma(h_n,t_n),
\end{align*}
because
\[|x(t)-x(t')| \le \int_t^{t'}|\dot x(s)|ds \le Ph_n\ \text{for all}\ t,t'\in[t_n,t_{n+1}].\]
\end{proof}

The estimate for the other semi-distance follows from the generalized Filippov Theorem \ref{continuous:filippov}.
\begin{corollary}
\label{convergence:Euler:1:2}
For every trajectory $\{y_n\}_n$ of the semi-implicit Euler scheme \eqref{Euler:1}, there exists
a solution $x(\cdot)$ of the ordinary differential inclusion \eqref{ODI} such that
\begin{equation}
\label{local:3}
|x(t_n)-y_n| \le  \int_0^{t_n} e^{\int_t^{t_n}l_f(s)+L_M(s) ds}\Gamma(|\fath|_\infty,t)dt
\end{equation}
for $n=0,\ldots,N$.
\end{corollary}

\begin{proof}
Consider a trajectory $\{y_n\}_n$ of the parameterized semi-implicit Euler scheme and its linear interpolation
\[ y(t) = y_n + (t-t_n)f(t_{n+1},y_{n+1}) + (t-t_n)m, \quad t\in[t_n,t_{n+1}], \]
where $m\in M(t_n,y_n)$. Then
\begin{align*}
g(t) &= \dist(\dot y(t),F(t,y(t)))\\
&= \dist(f(t_{n+1},y_{n+1})+m,f(t,y(t))+M(t,y(t)))\\
&\le |f(t_{n+1},y_{n+1})-f(t,y(t))| + \dist(M(t_n,y_n),M(t,y(t)))\\
&\le \Gamma(h_n,t)
\end{align*}
for all $t\in[t_n,t_{n+1}]$. As the mapping $x\mapsto F(t,x)$ is $(l_f(t)+L_M(t))$-ROSL, 
Theorem \ref{continuous:filippov} implies piecewise and
thus globally the existence of a solution $x(\cdot)$ of the differential inclusion \eqref{ODI}
satisfying \eqref{local:3}.
\end{proof}

\begin{remark}
\begin{itemize}
\item [a)] In view of \eqref{gamma:to:zero} and Corollaries \ref{convergence:Euler:1:1} and \ref{convergence:Euler:1:2},
\[\dist_H(\ES([0,T],x_0),\ES_\Phi(\Delta_{\fath},x_0)) \rightarrow 0\ \text{as}\ |\fath|_\infty\rightarrow 0,\]
which implies convergence
\[\dist_H(\reachable(T,x_0),\reachable_\Phi(T,x_0)) \rightarrow 0\ \text{as}\ |\fath|_\infty\rightarrow 0.\]
Moreover, error estimates \eqref{local:4} and \eqref{local:3} allow to exploit negative ROSL-constants
of the single-valued part $f$, but not of $M$, which is due to the semi-implicit construction of the scheme.
\item [b)] If $l_f(\cdot)\equiv l_f$, $L_M(\cdot)\equiv L_M$, and $h_n\equiv h$ are constant and 
$(t,x)\mapsto f(t,x)$ and $(t,x)\mapsto M(t,x)$ are $L$-Lipschitz, then estimate \eqref{local:5} simplifies to
\[|x(t_n)-y_n| \le 2L(1+P)he^{(\frac{l_f}{1-l_fh}+L_M)t_n}\]
and \eqref{local:3} becomes
\[|x(t_n)-y_n| \le \frac{2L(1+P)h}{l_f+L_M}(e^{(l_f+L_M)t_n}-1)\]
for every $n\in\{0,\ldots,N\}$.
\end{itemize}
\end{remark}

\subsection{Spatial discretization}
\label{single:step:discretization}

The parameterized semi-implicit Euler scheme analyzed in Sections \ref{single:step:properties} 
and \ref{single:step:dynamics} is discrete in time, but not in space. For a practical implementation, 
it is inevitable to introduce a spatial discretization, which is simple for explicit numerical schemes 
(see \cite{Beyn:Rieger:07}) but causes problems in the case of the fully implicit Euler scheme 
(see \cite{Beyn:Rieger:10}). The aim of this Section is to show that the parameterized semi-implicit 
Euler scheme can be discretized in a straight-forward way without losing its favorable properties.

\medskip

To this end, consider the grid $\Delta_\rho:=\rho\Z^d\subset\R^d$ for some $\rho>0$. The projection 
$P_\rho:\powerset(\R^d)\rightarrow\powerset(\Delta_\rho)$ from the subsets of $\R^d$ to the subsets
of $\Delta_\rho$
is defined by $P_\rho(A):=B_{\frac{\sqrt{d}}{2}\rho}(A)\cap\Delta_\rho$. 

\begin{lemma}
If $A\subset\R^d$ is nonempty, then
\[\dist_H(A,P_\rho(A)) \le \frac{\sqrt{d}}{2}\rho.\]
In particular, the set $P_\rho(A)$ is nonempty.
\end{lemma}
\begin{proof}
Let $\round:\R\rightarrow\Z$ be the usual rounding function. 
If $x\in A$, then the element $x^\rho\in\Delta_\rho$ specified by
\[x^\rho_n := \rho\cdot\round(\frac{x_n}{\rho}),\quad n=1,\ldots,d,\]
satisfies
\[|x-x^\rho| \le \sqrt{(\frac{\rho}{2})^2+\ldots+(\frac{\rho}{2})^2} = \frac{\sqrt{d}}{2}\rho,\]
so that $x^\rho\in B_{\frac{\sqrt{d}}{2}\rho}(A)\cap\Delta_\rho=P_\rho(A)$. As $x\in A$ was arbitrary,
\[\dist(A,P_\rho(A)) \le \frac{\sqrt{d}}{2}\rho.\]
Since $P_\rho(A)\subset B_{\frac{\sqrt{d}}{2}\rho}(A)$, it is clear that $\dist(P_\rho(A),A)\le\frac{\sqrt{d}}{2}\rho$.
\end{proof}

According to Lemma \ref{representation:lemma}, the image $\Phi(t,x;h)$ of the parameterized semi-implicit Euler scheme 
is a union \eqref{implicit:union} of solutions of implicit equations \eqref{implicit:equation} parameterized over
$M(t,x)$.
It is therefore natural to discretize the parameter set $M(t,x)$, solve the corresponding implicit
equations, and map the results to the spatial grid, which leads to the definition
\[\hat\Phi(t,x;h,\rho,\epsilon):=P_\rho(\{z\in x+hf(t+h,z)+hP_\epsilon(M(t,x))\}),\]
where $\rho>0$ and $\epsilon>0$ are the mesh sizes of grids in state and velocity space.
Set $\fatrho:=(\rho_0,\ldots,\rho_N)$ and $\fateps:=(\epsilon_0,\ldots,\epsilon_{N-1})$.

\begin{definition}
A sequence $\{\hat y_n\}_{n=0}^N\subset\R^d$ is called a trajectory of the fully discretized
parameterized semi-implicit Euler scheme $\hat\Phi$ associated with \eqref{ODI}
if 
\[\hat y_{n+1} \in \hat\Phi(t_n,\hat y_n;h_n,\rho_{n+1},\epsilon_n)\ \text{for}\ n=0,\ldots,N-1,
\quad \hat y_0\in P_{\rho_0}(\{x_0\}).\]
The set of all such trajectories is denoted $\ES_{\hat\Phi}(\Delta_{\fath},x_0)$.
Its dependence on $\fatrho$ and $\fateps$ is suppressed for the sake of readability.
\end{definition}

\begin{definition}
For any $t_n\in\Delta_{\fath}$, the reachable set $\reachable_{\hat\Phi}(t_n,x_0)$ of the fully discretized
semi-implicit Euler scheme is given by
\[\reachable_{\hat\Phi}(t_n,x_0)
:=\{\hat y_n\in\R^d: \{\hat y_n\}_{n=0}^N\in\ES_{\hat\Phi}(\Delta_{\fath},x_0)\}.\]
\end{definition}

\begin{proposition}
\label{discretized:convergence:1}
For every trajectory $\{y_n\}_{n=0}^N \in \ES_\Phi(\Delta_{\fath},x_0)$, there exists a trajectory 
$\{\hat y_n\}_{n=0}^N \in \ES_{\hat\Phi}(\Delta_{\fath},x_0)$ such that
\begin{align}
\label{perturbation:spatial:1}
|y_n-\hat y_n| &\le \frac{\sqrt{d}}{2}\rho_0\prod_{k=0}^{n-1}\frac{1+L_M(t_k)h_k}{1-l_f(t_{k+1})h_k} \\
&+ \frac{\sqrt{d}}{2}\sum_{k=0}^{n-1}\left(\prod_{j=k+1}^{n-1}\frac{1+L_M(t_j)h_j}{1-l_f(t_{j+1})h_j}\right)
(\rho_{k+1}+\frac{\epsilon_k h_k}{1-l_f(t_{k+1})h_k}) \nonumber
\end{align}
for all $n=0,\ldots,N$, and for every trajectory $\{\hat y_n\}_{n=0}^N \in \ES_{\hat\Phi}(\Delta_{\fath},x_0)$, 
there exists a trajectory $\{y_n\}_{n=0}^N \in \ES_\Phi(\Delta_{\fath},x_0)$
satisfying estimate \eqref{perturbation:spatial:1}.
\end{proposition}

\begin{proof}
The statement of the proposition follows by induction.
Let $\{y_n\}_{n=0}^N\in\ES_\Phi(\Delta_{\fath})$ be an arbitrary trajectory.
For $n=0$, take $\hat y_0\in P_{\rho_0}(y_0)$ such that
\[|y_0-\hat y_0| = \dist(\{x_0\},P_{\rho_0}(\{x_0\})) \le \frac{\sqrt{d}}{2}\rho_0.\]
Now assume that a trajectory $\{\hat y_k\}_{k=0}^n$, $n\in\{0,\ldots,N-1\}$, 
of \eqref{Euler:1} has been constructed that satisfies \eqref{perturbation:spatial:1}.
By Theorem \ref{properties:Phi}, there exists some $y\in\Phi(t_n,\hat y_n;h_n)$ with 
\[|y_{n+1}-y| \le \tfrac{1+L_M(t_n)h_n}{1-l_f(t_{n+1})h_n}|y_n-\hat y_n|.\]
By definition, there exists some $m\in M(t_n,\hat y_n)$ such that
\[ y \in \hat y_n + h_n f(t_{n+1},y) + h_nm,\]
and by the properties of the projection $P_{\epsilon_n}$, there exists a vector 
$\hat m\in P_{\epsilon_n}(M(t_n,\hat y_n))$ with $|m-\hat m|\le\frac{\sqrt{d}}{2}\epsilon_n$.
Let $z\in\R^d$ be the unique solution of the implicit equation
\[ z=\hat y_n+h_nf(t_{n+1},z)+h_n\hat m\]
and select some arbitrary $\hat y_{n+1}\in P_{\rho_{n+1}}(z)$.
Then $\hat y_{n+1}\in\hat\Phi(t_n,\hat y_n;h_n,\rho_{n+1},\epsilon_n)$, and it follows from
Lemma \ref{representation:lemma} that
\begin{align*}
&|y_{n+1}-\hat y_{n+1}| \le |y_{n+1}-y|+|y-z|+|z-\hat y_{n+1}|\\
&\le \tfrac{1+L_M(t_n)h_n}{1-l_f(t_{n+1})h_n}|y_n-\hat y_n| + \tfrac{\tfrac{\sqrt{d}}{2}\epsilon_n h_n}{1-l_f(t_{n+1})h_n} 
+ \tfrac{\sqrt{d}}{2}\rho_{n+1}\\
&\le \tfrac{\sqrt{d}}{2}\rho_0\prod_{k=0}^{n}\tfrac{1+L_M(t_k)h_k}{1-l_f(t_{k+1})h_k}
+ \tfrac{\sqrt{d}}{2}\sum_{k=0}^{n}\left(\prod_{j=k+1}^{n}\tfrac{1+L_M(t_j)h_j}{1-l_f(t_{j+1})h_j}\right)
(\rho_{k+1}+\tfrac{\epsilon_k h_k}{1-l_f(t_{k+1})h_k}),
\end{align*}
so that the sequence $\{\hat y_k\}_{k=0}^{n+1}$ satisfies \eqref{perturbation:spatial:1}.

The other direction can be shown with the same arguments.
\end{proof}

\begin{remark}
\label{discretization:remark}
\begin{itemize}
\item [a)] In view of Proposition \ref{discretized:convergence:1},
\[\dist_H(\ES_\Phi(\Delta_{\fath},x_0),\ES_{\hat\Phi}(\Delta_{\fath},x_0)) \rightarrow 0\]
when $\epsilon\rightarrow 0$ and $\rho=o(h)$.
All error estimates allow to exploit negative OSL-constants of $f$.
\item [b)] If $l_f(\cdot)\equiv l_f$, $L_M(\cdot)\equiv L_M$, $h$ and $\rho$ are constant, then estimate 
\eqref{perturbation:spatial:1} changes to
\begin{align*}
&\dist_H(\ES_\Phi(0,x_0,t_n),\ES_{\hat\Phi}(0,P_\rho(\{x_0\}),t_n)) \\
&\le \frac{\sqrt{d}}{2}\rho e^{\frac{l_ft_n}{1-l_fh}+L_Mt_n}
+ \frac{\sqrt{d}}{2}\frac{e^{L_Mt_n+\frac{l_ft_n}{1-l_fh}}-1}
{L_M+\frac{l_f}{1-l_fh}+L_M\frac{l_f}{1-l_fh}h}(\frac{\rho}{h}+\frac{\epsilon}{1-l_fh})
\end{align*}
for $n=0,\ldots,N$. 

The term $\frac{\sqrt{d}}{2}\rho\exp(\frac{l_ft_n}{1-l_fh}+L_Mt_n)$ originates from the projection of the initial value $x_0$
to the grid $\Delta_\rho$. If the grid is centered in $x_0$ instead of the origin, this error term vanishes.
This is impossible if the initial value is in fact an initial set that must be discretized. 
As discussed in \cite{Beyn:Rieger:07}, it is reasonable to choose $\rho=h^2$ and $\epsilon=h$ in order to 
obtain first-order convergence.
\item[c)] An implementation of the fully discretized scheme $\hat \Phi$ is much easier than that
of the implicit Euler scheme presented in \cite{Beyn:Rieger:10}, where explicit knowledge of the 
one-sided Lipschitz constant and the moduli of continuity of the right-hand side are required not
only for the error estimates, but for the practical computations.
Underestimation of these parameters can lead to divergence or failure of the implicit Euler scheme, 
whereas overestimation implies pessimistic estimates and smaller step-sizes than necessary
as usual. The semi-implicit scheme $\hat\Phi$ does not have any such flaws.
\end{itemize}
\end{remark}

\section{The semi-implicit split scheme}

In this section, a simpler version of the semi-implicit Euler scheme will be investigated.
One step of this method is the multivalued mapping $\Psi:D\rightarrow\powerset(\R^d)$
given by
\begin{subequations}
\begin{equation}
\label{direct:1}
\Psi(t,x;h):=z+hM(t,x),
\end{equation}
where $D$ is the domain defined in Section \ref{single:step:properties} and $z\in\R^d$ is the unique solution of the implicit equation 
\begin{equation}
\label{direct:2}
z=x+hf(t+h,z).
\end{equation}
\end{subequations}
The obvious advantage of this scheme is that only one implicit equation must be solved in every step instead 
of a whole family of such equations. This is achieved by fully separating the problem of solving implicit
equations and treating the set-valued part $M$. 

\medskip

The discussion of the semi-implicit split scheme will be concise. On one hand, the 
techniques are similar to those used in Section \ref{parametrized:sec}. On the other hand,
useful statements about the solution of implicit equation \eqref{direct:2}, which is in fact
the defining equation of the classical implicit Euler scheme, can be deduced from the results
in Section \ref{parametrized:sec} applied to the right-hand side $F(t,x)=\{f(t,x)\}$.

\medskip

By Lemma \ref{representation:lemma}, equation \eqref{direct:2} admits a unique solution $z$
when $(t,x,h)\in D$, and by Theorem \ref{Euler:1:continuous} and Assumption A2), the mapping
$(t,x,h)\mapsto\Psi(t,x;h)$ is jointly continuous and $x\mapsto\Psi(t,x;h)$ is Lipschitz 
continuous with constant $\frac{1}{1-l_f(t+h)h}+L_Mh$ on $D$. Convexity of the images $\Psi(t,x;h)$
is evident.

\medskip

Assumption A3) guarantees that the iterates of the numerical scheme are well-defined.
\begin{definition}
A sequence $\{y_n\}_{n=0}^{N}\subset\R^d$ is called a trajectory of the 
semi-implicit split scheme $\Psi$ associated with \eqref{ODI} if 
\begin{equation}
\label{def:trajectory:2}
y_{n+1} \in \Psi(t_n,y_n;h_n)\ \text{for}\ n=0,\ldots,N-1,\quad y_0=x_0.
\end{equation}
The set of all such trajectories is denoted by $\ES_\Psi(\Delta_{\fath},x_0)$.
\end{definition}

\begin{definition}
For any $t_n\in\Delta_{\fath}$, the reachable set $\reachable_\Psi$ of the semi-implicit split
scheme $\Psi$ is given by
\begin{equation}
\label{def:reachable:2}
\reachable_\Psi(t_n,x_0)=\{y_n: \{y_k\}_{k=0}^N\in\ES_\Psi(\Delta_{\fath},x_0)\}.
\end{equation}
\end{definition}
By the above, the reachable set $\reachable_\Psi(t_n,x_0)$ depends continuously on $(\fath,x_0)$ and is 
Lipschitz continuous in the initial value, because it is a composition of multimaps with this property.

\medskip

Boundedness of trajectories can be shown as in Lemma \ref{boundedness:lemma},
so that the moduli of continuity $\tau_f$, $\chi_f$, and $\tau_M$ and the 
constant $P$ are well-defined. 

\begin{proposition}
\label{convergence:Euler:2:1}
For any solution $x(\cdot)\in\ES([0,T],x_0)$, there exists a trajectory $\{y_n\}_{n=0}^N\in\ES_\Psi(\Delta_{\fath},x_0)$
such that
\begin{equation}
\label{local:8}
|x(t_n)-y_n| \le \sum_{k=0}^{n-1}(\prod_{j=k+1}^{n-1}(\frac{1}{1-l_f(t_{j+1})h_j}+L_M(t_j)h_j))
h_k\tilde\Gamma(h_k,t_k)
\end{equation}
for $n=0,\ldots,N$, where
\begin{align*}
\tilde \Gamma(h_k,t_k) &:=\frac{1}{1-l_f(t_{k+1})h_k}(\tau_f(h_k)+\chi_f(Ph_k))\\
& \quad +\chi_f(\int_{t_k}^{t_{k+1}}e^{\int_s^{t_{k+1}}l(u)du}P\,ds)+\tau_M(h_k)+L_M(t_k)Ph_k.
\end{align*}
\end{proposition}

\begin{proof}
The statement is obtained by induction. It is trivially satisfied for $n=0$.
Let $x(\cdot)$ be an arbitrary solution of \eqref{ODI}, and assume that a trajectory $\{y_k\}_{k=0}^{n}$ 
satisfying \eqref{local:8} has been constructed.
By definition of a solution, 
\[\dot x(t)=f(t,x(t))+m(t,x(t)),\quad x(0)=x_0\]
with $m(t,x(t))\in M(t,x(t))$.
Consider the solution $e(\cdot)$ of the auxiliary problem 
\[\dot e(t)=f(t,e(t)),\quad e(t_n)=x(t_n).\]
The estimate given in the proof of Theorem 3.2 in \cite{Donchev:Farkhi:98} with $F(t,x):=f(t,x)$ and
$s(t):=|x(t)-e(t)|$ yields 
\begin{align*}
|x(t_{n+1})-e(t_{n+1})| &\le \int_{t_n}^{t_{n+1}}e^{\int_s^{t_{n+1}}l(u)du}|m(s,x(s))|ds \\
&\le \int_{t_n}^{t_{n+1}}e^{\int_s^{t_{n+1}}l(u)du}Pds.
\end{align*}
Define $m:=\Proj(\frac{1}{h_n}\int_{t_n}^{t_{n+1}}m(s,x(s))ds,M(t_n,x(t_n)))$. Then $y:=z+h_nm\in\Psi(t_n,x(t_n),h_n)$,
where $z$ is the unique solution of 
\begin{equation}
\label{local:9}
0=x(t_n)+h_nf(t_{n+1},z)-z.
\end{equation}
By Theorem \ref{solvability:theorem} applied to \eqref{local:9} with initial guess $e(t_{n+1})$,
\begin{align*}
|e(t_{n+1})-z| &\le \frac{1}{1-l_f(t_{n+1})h_n}|x(t_n)+h_nf(t_{n+1},e(t_{n+1}))-e(t_{n+1})|\\
&\le \frac{1}{1-l_f(t_{n+1})h_n} \int_{t_n}^{t_{n+1}}|f(t_{n+1},e(t_{n+1}))-f(s,e(s))|ds\\
&\le \frac{h_n}{1-l_f(t_{n+1})h_n} (\tau_f(h_n) + \chi_f(Ph_n)),
\end{align*}
and
\begin{align*}
&|\int_{t_n}^{t_{n+1}}m(s,x(s))ds-h_nm| = \dist(\int_{t_n}^{t_{n+1}}m(s,x(s))ds,h_nM(t_n,x(t_n)))\\
&\le \int_{t_n}^{t_{n+1}}\dist(M(s,x(s)),M(t_n,x(t_n)))ds \le h_n(\tau_M(h_n)+L_M(t_n)Ph_n),
\end{align*}
so that
\begin{align*}
|x(t_{n+1})-y| &= |[x(t_{n})+\int_{t_n}^{t_{n+1}}f(s,x(s))+m(s,x(s))ds]-[z+h_nm]|\\
&\le |x(t_{n})+\int_{t_n}^{t_{n+1}}f(s,e(s))ds - z| + \int_{t_n}^{t_{n+1}}|f(s,x(s))-f(s,e(s))|ds \\
&\quad + |\int_{t_n}^{t_{n+1}}m(s,x(s))ds - h_nm|\ \le\ \tilde\Gamma(h_n,t_n),
\end{align*}
As the scheme $\Psi$ is Lipschitz, there exists an element $y_{n+1}\in\Psi(t_n,y_n;h_n)$ such that
\begin{align*}
|x(t_{n+1})-y_{n+1}| &\le |x(t_{n+1})-y| + |y-y_{n+1}| \\
&\le (\frac{1}{1-l_f(t_{n+1})h_n}+L_M(t_n)h_n)|x(t_n)-y_n| + \tilde\Gamma(h_n,t_n),
\end{align*}
and hence $y_{n+1}$ satisfies \eqref{local:8}.
\end{proof}

The convergence proof for the other semi-distance proceeds along the same lines as that of Corollary 
\ref{convergence:Euler:1:2}.

\begin{proposition}
\label{convergence:Euler:2:2}
For every trajectory $\{y_n\}_n\in\ES_\Psi(\Delta_{\fath},x_0)$, there exists
a solution $x(\cdot)\in\ES([0,T],x_0)$  such that
\begin{equation}
\label{local:10}
|x(t_n)-y_n| \le  \int_0^{t_n} e^{\int_t^{t_n}l_f(s)+L_M(s) ds}\Gamma(|\fath|_\infty,t)dt
\end{equation}
for $n=0,\ldots,N$.
\end{proposition}

\begin{remark}
\item [a)] In view of Propositions \ref{convergence:Euler:2:1} and \ref{convergence:Euler:2:2},
\[\dist_H(\ES([0,T],x_0),\ES_\Psi(\Delta_{\fath},x_0)) \rightarrow 0\ \text{as}\ |\fath|_\infty\rightarrow 0.\]
The estimates allow to exploit negative ROSL-constants.
\item [b)] If $l_f(\cdot)\equiv l_f$, $L_M(\cdot)\equiv L_M$, and $h_n\equiv h$ are constant and 
$(t,x)\mapsto f(t,x)$ and $(t,x)\mapsto M(t,x)$ are $L$-Lipschitz, then estimate \eqref{local:8} simplifies to
\begin{align*}
&|x(t_n)-y_n| \\
&\le\frac{\exp((\frac{l_f}{1-l_fh}+L_M)t_n)-1}{\frac{l_f}{1-l_fh}+L_M}
\left(\tfrac{L(1+P)h}{1-l_fh}+\tfrac{1}{l_f}(\exp(l_fh)-1)+(L+L_MP)h\right).
\end{align*}
\end{remark}

A straight-forward spatial discretization of the scheme $\Psi$ is given by
\[\hat\Psi(t,x;h,\rho):=P_\rho(\Psi(t,x;h)),\]
i.e.\ the solution $z$ of \eqref{direct:2} is computed and the set $z+M(t,x)$ 
is projected to the spatial grid. The concrete implementation of this process
depends on the implementation of the mapping $M$ and is similar to that of the explicit 
Euler scheme (see \cite{Beyn:Rieger:07}). It is therefore significantly more simple than 
the implementation of the fully implicit and the parameterized semi-implicit Euler schemes.

\begin{definition}
A sequence $\{\hat y_n\}_{n=0}^N\subset\R^d$ is called a trajectory of the fully discretized
semi-implicit split scheme $\hat\Psi$ associated with \eqref{ODI}
if 
\[\hat y_{n+1} \in \hat\Psi(t_n,\hat y_n;h_n,\rho_{n+1})\ \text{for}\ n=0,\ldots,N-1,
\quad \hat y_0\in P_{\rho_0}(\{x_0\}).\]
The set of all such trajectories is denoted $\ES_{\hat\Psi}(\Delta_{\fath},x_0)$.
\end{definition}

\begin{definition}
For any $t_n\in\Delta_{\fath}$, the reachable set $\reachable_{\hat\Psi}(t_n,x_0)$ of the fully discretized 
semi-implicit split scheme is given by
\[\reachable_{\hat\Psi}(t_n,x_0)
:=\{\hat y_n\in\R^d: \{\hat y_n\}_{n=0}^N\in\ES_{\hat\Psi}(\Delta_{\fath},x_0)\}.\]
\end{definition}

\noindent It is easy to see that for any $\{y_n\}_{n=0}^N\in\ES_\Psi(\Delta_{\fath},x_0)$,
there exists some $\{\hat y_n\}_{n=0}^N\in\ES_{\hat\Psi}(\Delta_{\fath},x_0)$ satisfying
\begin{equation}
\label{local:11}
|y_n-\hat y_n| \le \frac{\sqrt{d}}{2}\sum_{k=0}^n(\prod_{j=k}^{n-1} (\frac{1}{1-l_f(t_{j+1})h_j}+L_M(t_j)h_j))\rho_k,
\end{equation}
and for every $\{\hat y_n\}_{n=0}^N\in\ES_{\hat\Psi}(\Delta_{\fath},x_0)$, there exists some
$\{y_n\}_{n=0}^N\in\ES_\Psi(\Delta_{\fath},x_0)$ such that \eqref{local:11} holds.

If $l_f(\cdot)\equiv l_f$, $L_M(\cdot)\equiv L_M$, $h$ and $\rho$ are constant, then \eqref{local:11}
can be replaced with
\[|y_n-\hat y_n|  \le \frac{\sqrt{d}}{2}\frac{\exp(\frac{(n+1)l_fh}{1-l_fh}+(n+1)L_Mh)-1}{\frac{l_f}{1-l_fh}+L_M}\frac{\rho}{h}.\]
The unusual factor $(n+1)$ originates from the projection of the initial value to the spatial grid (cp.\ 
Remark \ref{discretization:remark}).

\section{Performance}

A comparison of the explicit and implicit Euler schemes on a stiff Michaelis-Menten system
was given in \cite{Beyn:Rieger:10}. As the simulations look very similar when the implicit Euler
scheme is replaced with one of the two semi-implicit schemes under discussion, no such graphics
are presented here. Instead, the performance of both fully discretized semi-implicit schemes is 
investigated when applied to the Dahlquist-like test inclusion
\begin{equation}
\label{Dahlquist}
\dot x(t) \in -x(t) + [-1,1],\quad x(0)=x_0\in\R,
\end{equation}
which seems to be an appropriate setting for testing implicit schemes.
Both methods are so much faster than the fully implicit Euler scheme from \cite{Beyn:Rieger:10},
which is in addition dangerously sensitive to ill-estmated constants, that a detailed comparison 
with this method is inadequate. 

The results are displayed in Figure \ref{aufwand}. The convergence of both methods $\hat \Phi$ and $\hat \Psi$
measured in terms of the errors
\begin{align*}
\max_{n\in\{0,\ldots,N_h\}}\dist_H(\reachable(t_n,5),\reachable_{\hat \Phi}(t_n,5)),\\
\max_{n\in\{0,\ldots,N_h\}}\dist_H(\reachable(t_n,5),\reachable_{\hat \Psi}(t_n,5))
\end{align*}
is linear in the constant step-size $|\fath|_\infty=5/N_h$, but not in the consumed time, 
which is typical for numerical
methods for differential inclusions (cp.\ \cite{Beyn:Rieger:07}). 
Due to its favorable analytical properties, the parameterized scheme is better when the
numerical errors are compared to the overall step-size. When performance is measured in 
computation time, however, the split scheme is far more efficient because of its simple
spatial discretization. The reasoning below shows that this performance gap will grow 
dramatically with the dimension of the state space.

\medskip

Assume that the step-size $h$ and the grid-width $\rho$ are constant and that $\epsilon=h$ is
fixed. The computational costs caused by one
time step of the parameterized and the split schemes can be roughly
expressed as
\begin{align*}
\text{time}_\text{par} &\approx C_\text{scan}\tfrac{\text{vol}(\text{domain})}{\rho^d}
+ (C_\text{Newton}(d)+C_\text{eval})\tfrac{\text{vol}(\text{image}(F))}{h^d}\tfrac{\text{vol}(\text{curr.\ state})}{\rho^d}\\
\text{time}_\text{split} &\approx C_\text{scan}\tfrac{\text{vol}(\text{domain})}{\rho^d}
+ \left(C_\text{Newton}(d) + C_\text{eval}\tfrac{\text{vol}(\text{image}(F))}{h^d}\right)
\tfrac{\text{vol}(\text{curr.\ state})}{\rho^d}
\end{align*}
with notation
\begin{itemize}
\item $C_\text{scan}$ -- time needed to check whether some grid point is an element of the current state 
\item $C_\text{Newton}(d)$ -- time needed to compute an approximate solution using Newton's method
(depends on space dimension $d$)
\item $C_\text{eval}$ -- time needed to evaluate and project the final result of one individual computation 
to the spatial grid.
\item domain -- the domain in $\R^d$ on which the algorithm computes the solution sets
\item curr.\ state -- the reachable set at present time.
\end{itemize}
Since $\rho=h^2$ and $C_\text{scan}<C_\text{eval}\ll C_\text{Newton}(d)$, it is evident that the consumed time
grows exponentially in $d$ and that
\[\text{time}_\text{par} \approx \frac{\text{vol}(\text{image}(F))}{h^d}\text{time}_\text{split}.\]
This rule of thumb is verified by Figure \ref{aufwand}.
As the error estimates for both schemes are linear in $|\fath|_\infty$, the parameterized semi-implicit Euler
scheme cannot compete with the split scheme. It has, however, one advantage that is illustrated 
in the subsequent analysis of the Dahlquist-like inclusion \eqref{Dahlquist}.

\begin{figure}[h]
\begin{center}
\includegraphics[scale=0.6]{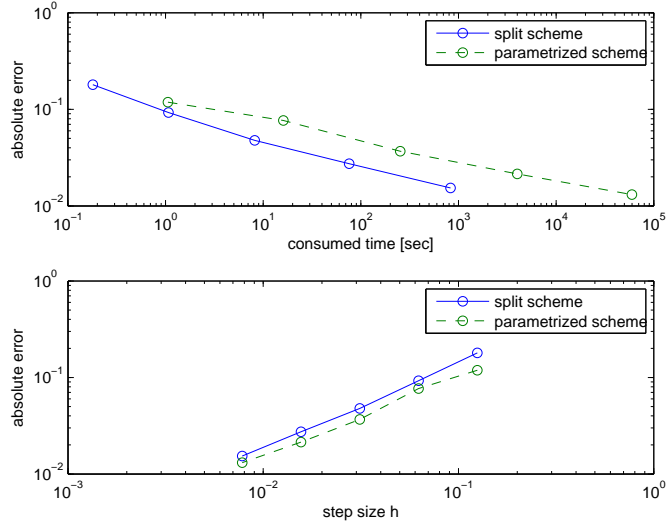}
\end{center}
\caption{Error analysis of both methods tested on the Dahlquist-like equation \eqref{Dahlquist}
with $x_0=5$ on the time interval $[0,5]$.
The convergence of both methods is linear in the constant step-size $h$, but not in the consumed time. \label{aufwand}}
\end{figure}

By monotonicity, inclusion \eqref{Dahlquist} admits upper and lower solutions 
\[x^{(+)}(t)=e^{-t}(x_0-1)+1, \quad x^{(-)}(t)=e^{-t}(x_0+1)-1,\]
and it is easy to see that the interval $[-1,1]$ is the global attractor for the 
multivalued flow induced by \eqref{Dahlquist}.
The parameterized semi-implicit Euler scheme has upper and lower solutions given by the recursions
\[y^{(+)}_{n+1} = \frac{y^{(+)}_n+h}{1+h}, \quad y^{(-)}_{n+1} = \frac{y^{(-)}_n-h}{1+h},\]
and its attractor coincides with that of the original inclusion.
The split scheme admits upper and lower solutions given by the recursions
\[\tilde y^{(+)}_{n+1} = \frac{\tilde y^{(+)}_n}{1+h}+h, \quad \tilde y^{(-)}_{n+1} = \frac{\tilde y^{(-)}_n}{1+h}-h,\]
and simple computations show that its attractor is $[-1-h,1+h]$. As a consequence, the parameterized 
semi-implicit Euler scheme seems to be superior with respect to correct asymptotic  behavior (in time).

\section{Conclusion}

The semi-implicit split scheme is at present the fastest numerical method applicable to stiff 
differential inclusions. 
In some sense, this paper finishes the discussion of explicit and implicit first-order methods 
of classical type, because combined with \cite{Beyn:Rieger:07} it provides a fairly clear picture
of what can be achieved with such schemes. 
Grid-based methods are currently the best we have, and they 
are far better than trajectory-based schemes, because they reduce the complexity in every step 
by projecting to the spatial grid and thus identifying many individual solutions. 

Nevertheless, the performance of these methods leaves much to be desired.
This is partly due to the fact that they are only approximations of first order, 
but it is mainly because of redundant computations that arise from the overlap of computed images 
and account for an overwhelming majority of the computational costs.

Two worthwile future challenges are therefore clearly defined: It is necessary to obtain a better
understanding of the fine structure of the solution set of the differential inclusion \eqref{ODI}
in order to be able to develop higher order schemes (of classical type) systematically.
Moreover, the invention of non-classical schemes that avoid redundant computations must be pushed 
forward. Numerical methods which discretize and track only the boundary of the reachable sets 
have been tested experimentally with very promising results. The dynamics of the boundary,
however, are complicated, and for that reason it is very difficult to prove error estimates
for this type of schemes.

\appendix

\section{Solvability and stability theorems}

The following solvability theorem is a modification of Corollary 3 in \cite{Beyn:Rieger:10}.
Every continuous set-valued mapping is upper semicontinuous.

\begin{theorem}
\label{solvability:theorem}
Let $G:\R^d\rightarrow\CoCpSets(\R^d)$ be upper semicontinuous and $l$-ROSL with $l<0$.
Then for any $y\in\R^d$, the set $S_G(y):=\{z\in\R^d: y\in G(z)\}$ is nonempty and compact 
and satisfies
\begin{equation}
\label{diameter:estimate}
\diam S_G(y) \le -\frac{1}{l}\sup_{x\in S(y)}\diam G(x).
\end{equation}
Moreover, the estimates
\begin{eqnarray}
\dist(x,S_G(y)) &\le& -\frac{1}{l}\dist(y,G(x)), \label{defect:estimate}\\
\dist(S_G(y),x) &\le& -\frac{1}{l}\dist(G(x),y) \label{distance:estimate}
\end{eqnarray}
hold for arbitrary $x\in\R^d$.
\end{theorem}

\begin{proof}
As $G$ is usc, $S_G(y)$ is closed for any $y\in\R^d$. Estimate \eqref{defect:estimate}
is Corollary 3 in \cite{Beyn:Rieger:10}. If $z\in S_G(y)$, then $y\in G(z)$, and by the 
ROSL property there exists some $\eta\in G(x)$ such that
\[-|y-\eta|\cdot |z-x| \le \langle y-\eta,z-x\rangle \le l|z-x|^2.\]
Hence 
\[|z-x| \le -\frac{1}{l}|y-\eta| \le -\frac{1}{l}\dist(G(x),y),\]
which proves \eqref{distance:estimate}. If, in addition, $x\in S_G(y)$, then $y\in G(x)$ and
\eqref{diameter:estimate} follows from \eqref{distance:estimate}.
\end{proof}

The stability Theorem below is cited from \cite[Theorem 13]{Beyn:Rieger:10}
\begin{theorem}
\label{continuous:filippov}
Let $l:[0,T]\rightarrow\R$ be continuous, and $F:[0,T]\rightarrow\CoCpSets(\R^d)$ be a jointly 
continuous set-valued mapping which is $l(t)$-ROSL in the second argument. Then for any given 
$y\in C^1([0,T],\R^d)$, there exists a solution $x(\cdot)$ of the differential inclusion
\begin{equation}
\label{standard:awa}
\dot x(t)\in F(t,x(t)),\ t\in[0,T], \quad x(0)=x_0
\end{equation}
such that
\begin{equation}\label{filippov:estimate}
|x(t)-y(t)| \leq e^{\int_{0}^tl(s)ds}|x(0)-y(0)| + \int_{0}^te^{\int_s^tl(\tau)d\tau}g(s)ds
\end{equation}
for all $t\in[0,T]$, where $g\in C([0,T],\R_+)$ is defined by
\begin{equation*}
g(t):=\dist(\dot y(t),F(t,y(t))).
\end{equation*}
If $l$, $F$ and $y$ are defined on $[0,\infty)$ and have the same properties as above,
then the above statement holds on the interval $[0,\infty)$.
\end{theorem}

\bibliographystyle{plain}
\bibliography{semi_implicit_Euler_ODIs}
\end{document}